\documentclass[preprint, 12pt, english]{amsart}
\usepackage{amsthm}
\usepackage{amsmath}
\usepackage{latexsym, amssymb}
\usepackage{txfonts}
\usepackage{mathtools}
\usepackage{color}
\usepackage[all]{xy}
\usepackage[normalem]{ulem} 
 \usepackage{array} 
\usepackage{hyperref}

\newtheorem{thm}{Theorem}[section] 

\newtheorem{cor}[thm]{Corollary}

\newtheorem{lem}[thm]{Lemma}
\newtheorem{prop}[thm]{Proposition}

\newtheorem{ques}[thm]{Question}

\theoremstyle{definition}

\DeclareMathOperator{\mychar}{char} %

\newcommand\operA[2]{{\if!#2!\operatorname{#1}\else{\operatorname{#1}_{#2}^{\phantom{I}}}\fi}} 

\def\Br{{\operatorname{Br}}}

\def\Z{\mathbb{Z}}

\newcommand{\Trace}[1][]{\if!#1!\operatorname{Tr}\else{\operatorname{Tr}_{#1}^{\phantom{I}}}\fi} 

\long\def\forget#1\forgotten{{}} %

\def\({\left(}
\def\){\right)}

\newcommand\LAY[3][]{{\begin{array}{c}\mbox{#2} \if#1!{}\else{+}\fi \\ \mbox{#3}\end{array}}}

\makeatletter

\def\ps@pprintTitle{%
 \let\@oddhead\@empty
 \let\@evenhead\@empty
 \def\@oddfoot{}%
 \let\@evenfoot\@oddfoot}

\newcommand{\bigperp}{%
  \mathop{\mathpalette\bigp@rp\relax}%
  \displaylimits
}

\newcommand{\bigp@rp}[2]{%
  \vcenter{
    \m@th\hbox{\scalebox{\ifx#1\displaystyle2.1\else1.5\fi}{$#1\perp$}}
  }%
}
\makeatother

\renewcommand{\geq}{\geqslant}
\renewcommand{\leq}{\leqslant}

\DeclareMathOperator{\Alg}{Alg}

\DeclareMathOperator{\ed}{ed}

\newif\iffurther
\furtherfalse

\begin{document}

\title[ED of CSA in Bad Characteristic]{Essential Dimension of Central Simple Algebras when the Characteristic is Bad}

\author{Adam Chapman}
\address{School of Computer Science, Academic College of Tel-Aviv-Yaffo, Rabenu Yeruham St., P.O.B 8401 Yaffo, 6818211, Israel}
\email{adam1chapman@yahoo.com}

\author{Kelly McKinnie}
\address{Department of Mathematics, University of Montana, Missoula, MT 59812, USA}
\email{kelly.mckinnie@mso.umt.edu}

\begin{abstract}
This is a survey of the existing literature, the state of the art, and a few minor new results and open questions regarding the essential dimension of central simple algebras and finite sequences of such algebras over fields whose characteristic divides the degree of the algebras under discussion. Upper and lower bounds as well as a few precise evaluations of this dimension are included.
\end{abstract}

\keywords{
Essential Dimension; Central Simple Algebras; Fields of Positive Characteristic}
\subjclass[2020]{16K20 (primary); 16K50, 19D45 (secondary)}
\maketitle

\section{Introduction}

Fix a prime number $p$ and an algebraically closed field $k$. Let $\text{Fields}_k$ be the category of fields containing $k$. 
Given a covariant functor $\mathcal{F}$ from $\text{Fields}_k$ to the category of sets, we define the essential dimension of any $A$ in $\mathcal{F}(F)$, denoted $\operatorname{ed}(A)$, to be the minimal possible $d$ for which there exists $E \in \text{Fields}_k$ with $E \subseteq F$ such that $\operatorname{tr.deg}(E/k)=d$ and $A=B_F$ where $B \in \mathcal{F}(E)$. 
The essential dimension of $\mathcal{F}$, denoted $\ed(\mathcal{F})$, is the supremum on the essential dimension of $A$, where $A$ ranges over all elements in $\mathcal{F}(F)$ and $F$ ranges over all the fields containing $k$.
The essential $p$-dimension of $A$ in $\mathcal{F}(F)$, denoted $\operatorname{ed}_p(A)$, is defined to be the minimum on the essential dimension of $A_L$, where $L$ ranges over all field extensions of $F$ of degree prime to $p$.
The essential $p$-dimension of $\mathcal{F}$ is thus defined to be the supremum on the essential $p$-dimension of $A$, where $A$ ranges over all elements in $\mathcal{F}(F)$ and $F$ ranges over all the fields containing $k$. In particular, $\ed_p(\mathcal F)\leq \ed(\mathcal F)$.

The essential $p$-dimension is closely linked to the concept of $p$-special fields. Recall that a field $F$ is $p$-special if it has no finite extension fields of degree prime to $p$ \cite[Section 101.B]{EKM}. 
In \cite{Pfister:1995} a somewhat different definition of a $p$-special field appears: a field which has only field extensions of degree a power of $p$. The definitions coincide (see \cite[Lemma 2.1]{Lotscher:2013}). Thus, a $p$-special field $K$ is obtained as the minimal field extension of $F$ inside its algebraic closure $\bar{F}$ that has no proper prime-to-$p$ extension. In particular, $\operatorname{ed}_p(A)=\operatorname{ed}(A_K)$ for any $A \in \mathcal{F}(F)$. 

The main functor under discussion in this survey is $\Alg_{d,e}$ which is the functor mapping each field $F \supset k$ to the set of isomorphism classes of central simple algebras of degree $d$ over $F$ and exponent dividing $e$.  One is said to be studying the `bad characteristic case' if $p = \operatorname{char}(k)>0$ divides $d$, the degree of the central simple algebras and, conversely, we are in the `good characteristic' case if $p\nmid d$. 

The notion of essential dimension was initially defined by Reichstein and introduced in a series of seminal papers, including \cite{Buhler-Reichstein:1997}, \cite{Reichstein:2000} and \cite{ReichsteinYoussin:2000} (the latter is where the essential $p$-dimension was introduced). The essential dimension of $\operatorname{Alg}_{d,e}$ has received substantial attention in the literature, mostly due to its connection to linear algebraic groups of type $\text{A}$. But in general, its value is known only in a handful of cases, and in most cases we are very far from knowing its precise value.

There have been several surveys written on essential dimension over the last 15 years. Comprehensive surveys on the essential dimension of various linear algebraic groups were written by Merkurjev \cite{Merkurjev:2009}, \cite{Merkurjev:2013} and again in \cite{Merkurjev:2017}. A survey on central simple algebras, including a section on essential dimension, was written by Auel, Brussel, Garibaldi and Vishne \cite{ABGV}. 


This survey focuses on the bounds and values of the essential dimension of $\operatorname{Alg}_{d,e}$ and some related functors in the bad characteristic. The essential dimension in the bad case is often different from the good case. For example, if one considers the functor $\mathcal{F}={_p\Br}$ mapping a field $F$ to the $p$-torsion ${_p\Br(F)}$ of $F$, then its essential $p$-dimension is infinite in the good case, but in the bad  case, it is equal to 2. This is a result of the fact that in the bad case, every algebra of exponent $p$ is Brauer equivalent to a cyclic algebra of degree $p^m$ for some $m$, and that the essential $p$-dimension of $H_{p^m}^1$ in the bad case is 1. See Section \ref{cohomology} for a description of the functor $H^1_{p^m}$ and \cite[Props 7.2, 9.2]{OfekReichstein:2024} for further details of the argument.

The survey is organized as follows: we begin with the basic background on central simple algebras and Brauer groups. Then, we introduce the description of the Brauer group in terms of differential forms (which produce the Kato-Milne cohomology groups analogous to the Galois cohomology groups in the usual case). We deduce lower bounds for the essential dimension from this setting. We then turn to the known upper bounds for the essential dimension of $\operatorname{Alg}_{2^n,2}$, with a new upper bound for $\operatorname{Alg}_{8,2}$. 

Part of this survey (Section \ref{sequence}) is dedicated to the essential dimension of sequences of linked cyclic $p$-algebras, where we cover the known results from the literature, in addition to translations of existing results not originally set in the context of essential dimension, to the language of essential dimension (see Section \ref{PrimeExponent} for the application of an observation from \cite{ChapmanQueguiner:2024} to the essential dimension of tensor products of cyclic algebras of degree $p$). We also include a few words on the essential dimension of the functor $\operatorname{Dec}_{p^m,n}$ and extend certain results from \cite{McKinnie:2017} to the non-prime case.

Table \ref{table} records the known essential dimension bounds and values in the bad characteristic case with the corresponding good characteristic bounds provided as a comparison. In the first column the functors are defined over fields containing a fixed field $k$ of characteristic $p$ and $m$ and $n$ are positive integers. Relevant citations are given in the third column and the last column indicates the section(s) of this survey where the result is discussed or improved.

\begin{table}[htbp] 
        \centering 
        \caption{Essential Dimension values and bounds in the bad and good characteristic cases.}
        \label{table}
\begin{tabular}{|m{.18\textwidth}|m{.17\textwidth}|m{.17\textwidth}|m{.25\textwidth}|m{.085\textwidth}|} 
    \hline
    Functor & bound/value\newline in bad char& bound/value\newline in good char & Citations&Sec. \\
	\hline
$\ed(\operatorname{Alg}_{n,n})$ & $\leq n^2-3n+1$ & same & \cite[1.4]{GaribaldiGuralnick:2016}, \newline \cite{Lemire:2004}&\ref{d4}\\
    \hline
    $\ed(\Alg_{4,2})$,\newline $\ed_2(\Alg_{4,2})$ & 3 & 4 & \cite[Thm 1.1]{Baek:2011}, \newline \cite[Rmk 8.2]{BaekMerkurjev:2012}& \ref{d4}\\
    \hline
    $\ed(\Alg_{4,4})$, \newline $\ed_2(\Alg_{4,4})$ & $\geq 4$ and $\leq 5$ & 5 &\cite{Baek:2017},\newline \cite[Page 470]{Merkurjev:2013}& \ref{d4}\\
    \hline
    $\ed (\Alg_{8,2})$, \newline $\ed_2(\Alg_{8,2})$ & $\geq 4$ and $\leq 8$ & 8 & \cite{McKinnie:2017} \& \newline this survey (resp),\newline \cite[Cor 8.3]{BaekMerkurjev:2012}& \ref{d8}\\
    \hline
    $\ed((\Z/p\Z)^{\times n})$ & 1 & $n$ & \cite{Ledet:2004}&\ref{d8}, \ref{PrimeExponent}\\
    \hline
    $\ed(\operatorname{LCA}_{p,2})$ & 3 & open & \cite{Chapman:2024} & \ref{sequence}\\
    \hline
    $\ed(\operatorname{LCA}_{2,n})$, \newline $\ed_2(\operatorname{LCA}_{2,n})$ & $n+1$ & same &  this survey& \ref{sequence}\\
    \hline
    $\ed_p(\operatorname{Dec}_{p^m,n})$ & $n+1$ & $2n$ & this survey, \newline \cite[Ex 3.7]{Merkurjev:2013} & \ref{dec}\\
    \hline
$\ed_p(\Alg_{p^n,p})$ & $\geq n+1$ &$\geq 2n$& \cite{McKinnie:2017}, \newline \cite{ChapmanMcKinnie:2020},\newline \cite[Ex 3.7]{Merkurjev:2013}&\ref{cohomology}, \ref{PrimeExponent}\\
    \hline
    $\ed(H^1_p)$,\newline $\ed(H^1_{p^2})$&1, 2 (resp)&1&\cite{Ledet:2004}&\ref{cohomology}\\
    \hline
    $\ed_p(H^1_{p^n})$&1&1&\cite{ReichsteinVistoli:2018},\newline this survey\newline when $n=2$&\ref{cohomology}\\
    \hline
   
    $\ed_p(\Alg_{p,p})$ & 2 & same & \cite[Lem 8.5 (7)]{ReichsteinYoussin:2000}&\ref{primedegree}\\
    \hline
    $\ed(\Alg_{2n,2})$,\newline $n>4$ even & $\leq 2n^2-3n-6$ & $\leq 2n^2-3n-4$ & \cite[Thm 1.3]{GaribaldiGuralnick:2016}& \ref{PrimeExponent}\\
    \hline 
    $\ed_2(\Alg_{2^n,2})$ & $\leq 2^{2n-2}$ & $2^{2n-4}+2^{n-1}$ & \cite[Thm 1.1]{Baek:2012}&\ref{PrimeExponent}\\
    \hline
    \end{tabular}
    \end{table}
    
\section{Central Simple Algebras}

An associative algebra $A$ over a field $F$ is called a central simple $F$-algebra if it is simple (i.e., has no nontrivial two-sided ideals), its center is $F$ and its dimension over $F$ is finite.
In this case, $[A:F]$ is $d^2$ for some positive integer $d$ called the ``degree" of $A$, denoted $\deg A$.
In particular, $A \otimes_F \overline{F}$ is isomorphic to $M_d(\overline{F})$ where $\overline F$ is an algebraic closure of $F$. This embedding gives rise to the notion of reduced trace $\mathrm{Trd} : A \rightarrow F$, defined to be the trace of an element of $A$ via this isomorphism, and the reduced norm $\mathrm{Nrd} : A \rightarrow F$, defined to be the determinant of $A$ via this isomorphism. Both take values in $F$ despite the embedding into $M_d(\overline{F})$.
More generally, through this embedding, every element $a \in A$ has a characteristic polynomial
$f_a(x)=x^d-f_1(a) x^{d-1}-\dots-f_{d-1}(a) x-f_d(a)$ where each $f_n$ is a homogeneous form from $A$ to $F$ of degree $n$. The forms $f_1,\dots,f_d$ are the characteristic coefficients, and in particular, $f_1=\mathrm{Trd}$ and $f_d=\mathrm{Nrd}$. In particular $f_a(a)=0$.

Wedderburn proved that given a central simple algebra $A$ over $F$ there exists a division algebra $D$ with center $F$, unique up to isomorphism, and positive integer $t$ so that $A$ decomposes as $A \cong D \otimes M_t(F)$. 
$D$ is also a central simple $F$-algebra, and thus has a degree. The index of $A$, denoted $\operatorname{ind}(A)$,  is defined to be the degree of $D$.

This decomposition gives rise to the notion of Brauer equivalence: two central simple $F$-algebras are Brauer equivalent if their corresponding division algebras in Wedderburn's decomposition are isomorphic.
The set of Brauer classes of central simple algebras over a given field $F$ turns out to be a group with respect to the tensor product (over $F$) as the binary group operation.
This group, denoted $\Br(F)$ is called the ``Brauer group" of $F$.

The Brauer group $\Br(F)$ of $F$ is a torsion group, i.e., for every central simple algebra $A$ over $F$ there exists a positive integer $e$ such that $\underbrace{A \otimes \dots \otimes A}_{e \ \text{times}}$ is isomorphic to a matrix algebra over $F$. The minimal such $e$ is called the ``exponent" (or ``period" or ``order") of $A$, denoted $\exp(A)$. 
It is known that $\exp(A) | \operatorname{ind}(A)$, and that $\operatorname{ind}(A) | \exp(A)^q$ for some positive integer $q$. The minimal $q$ for which $\operatorname{ind}(A) | \exp(A)^q$ for all $A \in \Br(F)$ is called ``the Brauer dimension of $F$".

One important aspect of central simple algebras is that if the degree $d$ of $A$ is $m\cdot n$ with $\gcd(m,n)=1$, then $A \cong B \otimes C$ where $B$ and $C$ are central simple $F$-algebras with $\deg B=m$ and $\deg C=n$.
This motivates the study of the $p^m$-torsion subgroups of $\Br(F)$ for different prime integers $p$ and positive integers $m$.  These subgroups are denoted ${}_{p^m}\Br(F)$. Though this notion is independent of the characteristic, we focus on the case where $p$ is also the characteristic of the field $F$ as this survey is dedicated to the so-called bad characteristic case, which is unique in many respects. Algebras whose Brauer class is in ${}_{p^m}\Br(F)$ with $\mychar(F)=p$ are called ``$p$-algebras''.  The structure of $p$-algebras was developed by Albert in the 1930s, culminating with the result that every $p$-algebra is `cyclically representable', that is, equivalent in $\Br(F)$ to a cyclic algebra (\cite[Chapter VII, Theorem 31]{Albert:1968}).

There is a vast literature on central simple algebras, several references for the general theory include Gille and Szamuely's book \cite{GilleSzamuely:2017}, Berhuy and Oggier's book \cite{BerhuyOggier:2013}, Rowen's book \cite[Ch. 24]{Rowen:2008} and to some extent Albert's book \cite{Albert:1968}, although the last is an old reference and much of the terminology has changed since its publication.




\section{Algebras of Degree 4}\label{d4}
In this section, $p=2$, and so $k$ is an algebraically closed field of $\operatorname{char}(k)=2$.
For a field $F$ containing $k$, any central simple $F$-algebra of degree 4 is either of exponent 4, 2 or 1 (the split case). The case of algebras of exponent 2 is well-understood. 
\begin{prop}\label{Alg42}
The essential dimension and essential 2-dimension of $\operatorname{Alg}_{4,2}$ is exactly 3.
\end{prop}

\begin{proof}
Let $A$ be a central simple algebra of degree 4 and exponent 2 over $F$. By \cite[Page 174, Theorem 2]{Albert:1968}, $A$ is isomorphic to a tensor product of two quaternion algebras. On the one hand, the essential dimension of $A$ is at most 3 by \cite{McKinnie:2017}. On the other hand, let $A$ be a division algebra over $F$. For example, the algebra $[\alpha^{-1},\beta)_{2,F} \otimes [\gamma,\alpha)_{2,F}$ over $F=k(\alpha,\beta,\gamma)$ is division (see \cite{Chapman:2020}). $A$ remains division after any odd degree extension and any such extension cannot descend to a field of transcendence degree 2 over an algebraically closed field, because such fields are linked (see \cite[Theorem 3.1]{Baeza:1982} where the statement is about fields of $u$-invariant $\leq 4$ which applies because the $u$-invariant in this case is 4 by \cite[Prop 2]{MammoneMoresiWadsworth:1991}), i.e., every two quaternion algebras share a common maximal subfield, and then every biquaternion algebra over such fields is never a division algebra. Therefore $\operatorname{ed}(\operatorname{Alg}_{4,2})=3$.

\end{proof}

Compare this result to the lower bound on the essential dimension of $\Alg_{p^n,p}$ given in Proposition \ref{Algpnp}.
In \cite{Baek:2017}, it was shown that $4 \leq \operatorname{ed}(\operatorname{Alg_{4,4}}) \leq 5$.
Here is a short explanation of this result:
\begin{prop}
The essential dimension and essential 2-dimension of $\operatorname{Alg}_{4,4}$ are at least 4 and at most 5.
\end{prop}
\begin{proof} There is the invariant from \cite{Tignol:2006} that assigns a quadratic 4-fold Pfister form to any central simple algebra of degree 4 over a field of characteristic 2, which is hyperbolic if and only if the underlying algebra is cyclic. If the essential dimension of $\operatorname{Alg}_{4,4}$ were $\leq 3$, that would produce a 4-fold Pfister form over a field of transcendence degree $\leq 3$ over an algebraically closed field. Such forms are always hyperbolic  (because the $u$-invariant of such fields is at most $8$), and that would mean that every degree 4 algebra is cyclic. However, non-cyclic algebras of degree 4 exist (see \cite{Albert:1932}). Therefore, $\operatorname{ed}(\operatorname{Alg}_{4,4})\geq 4$.
For $\operatorname{ed}_2(\operatorname{Alg}_{4,4})\geq 4$ it is enough to add that anisotropic quadratic 4-fold Pfister forms remain anisotropic under odd extensions.

The upper bound for the essential dimension of $\operatorname{Alg}_{n,n}$ from \cite[Theorem 1.4]{GaribaldiGuralnick:2016} is $n^2-3n+1$, which for $n=4$ becomes 5. Since $\operatorname{ed}_2(\operatorname{Alg}_{4,4})\leq \operatorname{ed}(\operatorname{Alg}_{4,4})$, the upper bound for the essential 2-dimension follows as well.
\end{proof}

\section{Algebras of Degree 8 and Exponent 2}\label{d8}
The essential dimension of $\operatorname{Alg}_{8,2}$ has yet to be precisely evaluated.
In \cite{Baek:2011}, a lower bound of $3$ was provided, which was improved to 4 in \cite{McKinnie:2017}.
An upper bound of 10 was provided in \cite{Baek:2011}, but the argument contained an error (which, if fixed, increases the upper bound to 12). This bound was improved to 9 in \cite{Chapman:2023}. The computation made use of the essential dimension of $(\mathbb{Z}/p\mathbb{Z})^{\times n}$ over fields of characteristic $p$, which is equal to 1 (\cite{Ledet:2004}). The definition of the essential dimension of a finite group appears in \cite{Buhler-Reichstein:1997} and one can consult \cite{Ledet:2004} for the specific case of a finite $p$-group over a field of characteristic $p$. In short, the essential dimension of a finite group $G$ is the essential dimension of Galois extensions of $F$ with Galois group $G$, and in particular, $\operatorname{ed}(\mathbb{Z}/p^m\Z)=\operatorname{ed}(H^1_{p^m})$ (see Section \ref{cohomology}).

In this section, we recall the discussion from \cite{Chapman:2023} on the essential dimension of central simple algebras of degree $8$ and exponent 2 over fields of characteristic 2, and improve the upper bound from \cite{Chapman:2023} with a small modification of the argument, which avoids using the essential dimension of $(\mathbb{Z}/p\mathbb{Z})^{\times n}$.

\begin{lem}[{cf. \cite[Lemma 3.3]{Baek:2011}}]\label{Baeklike}
Suppose $F$ is a field of $\operatorname{char}(F)=2$, $K=F[i : i^2+i=\alpha]$ a separable quadratic field extension of $F$ and $Q=[a,bi+c)_{2,K}$ a quaternion algebra over $K$ with trivial corestriction to $F$, where $a,b,c \in F$. Then $Q$ descends to $F$.
\end{lem}

\begin{proof}
If $Q$ is split, the statement follows trivially. Suppose $Q$ is not split.
The corestriction of $Q$ is trivial, which means $[a,b^2\alpha+bc+c^2)_{2,F}$ is split (see \cite{MammoneMerkurjev:1991} for background on the corestriction of symbol algebras in the bad characteristic). Therefore, $a \equiv x^2 (b^2\alpha+bc+c^2) \pmod{\wp(F)}$ for some $x \in F^\times$.
Hence, $Q=[a,bi+c)_{2,K}=[x^2 (b^2\alpha+bc+c^2),bi+c)_{2,K}
=[x^2 (b^2(i^2+i)+bc+c^2),bi+c)_{2,K}=[xbi+x^2b^2i+x^2bc+xc),bi+c)_{2,K}=[(bi+c)(x+x^2b),bi+c)_{2,K}=[(bi+c)(x+x^2b),x+x^2b)_{2,K}=[a,x+x^2b)_{2,K}$.
\end{proof}


\begin{thm}\label{EDD8E2}
The essential dimension of a central simple algebra of degree 8 and exponent 2 over a field of characteristic 2 is at most 8.
\end{thm}

\begin{proof}
Given such an algebra $A$ over a field $F$ containing the base field $k$ of characteristic 2, we know that $A$ contains a subfield $F[x,y,z : x^2+x=\alpha,y^2+y=\beta,z^2+z=\gamma]$ for some $\alpha,\beta,\gamma \in F$ by \cite{Rowen:1984}. 
The centralizer $B$ of $L=F[x]$ in $A$ is thus $[\beta,ax+b)_{2,L} \otimes [\gamma,cx+d)_{2,L}$.
If $a=0$, the algebra $A$ decomposes as a tensor product of three quaternion algebras, and this case was covered already in \cite{McKinnie:2017}.
Suppose $a \neq 0$.
Then, by replacing $x$ with $x+\frac{b}{a}$, one may assume that $b=0$.
Now, the corestriction of $B$ back to $F$ is trivial, so $[\beta,\alpha a^2)_{2,F}=[\gamma,\alpha c^2+cd+d^2)_{2,F}$.
By the chain lemma for quaternion algebras, there exists $e\in F$ for which $[\beta,\alpha a^2)_{2,F}=[e,\alpha a^2)_{2,F}=[e,\alpha c^2+cd+d^2)_{2,F}=[\gamma,\alpha c^2+cd+d^2)_{2,F}$.
Therefore, there exist $r,s,t,\lambda,\tau,\theta \in F$ for which $\beta+e=\lambda^2+\lambda+r^2 a^2 \alpha$, $e=\tau^2+\tau+s^2 \alpha a^2(\alpha c^2+cd+d^2)$ and $\gamma+e=\theta^2+\theta+t^2 (\alpha c^2+cd+d^2)$.
By replacing $e$ with $e+\tau^2+\tau$, $\beta$ with $\beta+\lambda^2+\lambda+\tau^2+\tau$ and $\gamma$ with $\gamma+\theta^2+\theta+\tau^2+\tau$, we get rid of the occurrences of $\lambda,\tau,\theta$.
Suppose $r,s,t \neq 0$ (when one of the elements $r,s,t$ is zero, the argument is easier, see \cite{Chapman:2023}).
So 
\begin{eqnarray*}
B&=&[\beta,ax+b)_{2,L} \otimes [\gamma,cx+d)_{2,L}\\
&\sim_{Br}& [\beta+e,ax+b)_{2,L} \otimes [\gamma+e,cx+d)_{2,L} \otimes [e,(ax+b)(cx+d))_{2,L}\\
&\sim_{Br}& [\beta+e,\frac{1}{r}+a)_{2,L} \otimes [\gamma+e,\frac{1}{t}+c)_{2,L} \otimes [e,\frac{1}{s}+ac+ad)_{2,L}
\end{eqnarray*}
and $A\sim_{Br}[\alpha,f)_{2,F} \otimes [\beta+e,\frac{1}{r}+a)_{2,F} \otimes [\gamma+e,\frac{1}{t}+c)_{2,F} \otimes [e,\frac{1}{s}+ac+ad)_{2,F}$ (the descent computations here follow Lemma \ref{Baeklike}). The proof of why the last tensor product of 4 quaternion algebras is actually of index 8 is the same as in \cite{Chapman:2023}, and thus $A$ descends to a central simple algebra of degree 8 and exponent 2 over the field $k(\alpha,\beta,\gamma,a,c,d,r,f)[s,t : s^2 \alpha a^2(\alpha c^2+cd+d^2)=\beta+r^2 a^2 \alpha=\gamma+t^2 (\alpha c^2+cd+d^2)]$, which is of transcendence degree at most 8 over $k$.
\end{proof}

It was pointed out to us by J.P. Tignol that the argument presented in the proof of \ref{EDD8E2} follows closely to the argument of Rowen's in \cite[Theorem 2']{Rowen:1984}.

\section{Sequences of Linked Cyclic Algebras}\label{sequence}

Here we consider cyclic algebras of degree $p^m$ over $F$ of $\operatorname{char}(F)=p$.
Such a cyclic algebra is denoted $[\omega,\beta)_{p^m,F}$ for some  $\beta \in F^\times$ and $\omega=(\omega_1,\dots,\omega_m)\in W_m F$ where $W_m F$ is the ring of truncated Witt vectors of length $m$ over $F$. 

As a set, $W_m F = \{(x_0,\ldots,x_{m-1})\,|\, x_i \in F\}$, but addition and multiplication are defined by Witt Polynomials making them into a ring (\cite[Thm.8.30]{Jacobson} or \cite[pg.22]{Izhboldin:2000}). For example, $W_n = \sum_{i=0}^n p^iX_i^{p^{n-i}} \in \Z[X_0,\ldots]$ is the $n$-th Witt polynomial and the sum of two truncated Witt vectors $x=(x_1,\ldots,x_m)$ and $y=(y_1,\ldots,y_m)$ is defined to be $(s_0(x,y),s_1(x,y),\ldots,s_m(x,y))$ where $s_i \in \Z[X_0,\ldots,Y_0,\ldots]$ is uniquely defined so that $W_n(s_0(x,y),\ldots) = W_n(x)+W_n(y)$ for all $n\geq 0$. In particular, $s_0(x,y) = x_0+y_0$ and $s_1(x,y) = x_1+y_1+(x_0^p+y_0^p-(x_0+y_0)^p)/p$.





The algebra $[\omega,\beta)_{p^m,F}$ is generated by $u_1,\dots,u_m$ and $y$ subject to the relations $(u_1^p,\dots,u_m^p)-(u_1,\dots,u_m)=(\omega_1,\dots,\omega_m)$, and $(y u_1 y^{-1},\dots,y u_m y^{-1})=(u_1,\dots,u_m)+(1,0,\dots,0)$, where the arithmetic operations follow the rules of Witt vectors and $y^{p^m}=\beta$.

It is important to note that these algebras generate ${_{p^m}\Br}(F)$, and that a similar description in terms of differential forms (over $W_m F$) exists (see \cite{AravireJacobORyan:2018}). In \cite[Theorem 5.4]{ChapmanMcKinnie:2020}) it was shown that $\operatorname{ed}(\operatorname{Alg}_{p^{mn},p^m}) \geq n+1$.

In this section we are interested in the functor $\operatorname{LCA}_{p^m,n}$ which maps each $F$ containing $k$ of $\operatorname{char}(k)=p>0$ to length $n$ sequences of cyclically linked cyclic algebras of degree $p^m$ over $F$, i.e., sequences of the form $$([\omega,\beta_1)_{p^m,F},[\omega,\beta_2)_{p^m,F},\dots,[\omega,\beta_n)_{p^m,F})$$ for some $\omega \in W_m F$ and $\beta_1,\dots,\beta_n\in F^\times$.

The special case of $\operatorname{LCA}_{2,2}$ provides a characteristic 2 analogue to the functor studied in \cite[Theorem 1.3 (c)]{CerneleReichstein:2015} of triples of quaternion algebras with trivial tensor product over fields of characteristic 0: each $(Q_1,Q_2,Q_3)$ with trivial $Q_1 \otimes Q_2 \otimes Q_3$ corresponds to a linked pair $(Q_1,Q_2)$ where $Q_3$ is Brauer equivalent to $Q_1 \otimes Q_2$. The essential dimension of $\operatorname{LCA}_{2,2}$ was calculated in \cite{Chapman:2024} to be 3, in accordance with the analogous result from \cite[Theorem 1.3 (c)]{CerneleReichstein:2015}. In \cite{Chapman:2024} it was also shown that $\operatorname{ed}(\operatorname{LCA}_{p,2})=3$.

One key observation is that inseparably linked sequences have small essential dimension.
We say that $(C_1,\dots,C_n) \in \operatorname{LCA}_{p^m,n}$ are inseparably linked if they can be written as $C_1=[\omega_1,\beta)_{p^m,F},\dots,C_n=[\omega_n,\beta)_{p^m,F}$ for a fixed $\beta \in F^\times$.
In \cite[Theorem 4.7]{ChapmanFlorenceMcKinnie:2023}, it was proven that in this case, one can choose $\gamma_1,\dots,\gamma_{n-1} \in k$ and $\tau \in W_m F$ such that $C_1=[\tau,\beta)_{p^m,F}, C_2=[\tau,\beta+\gamma_1)_{p^m,F},\dots, C_n=[\tau,\beta+\gamma_{n-1})_{p^m,F}$.

\begin{cor}
The essential dimension of an inseparably linked sequence $(C_1,\dots,C_n) \in \operatorname{LCA}_{p^m,n}$ is at most $m+1$, and the essential $p$-dimension is at most 2.
\end{cor}

\begin{proof}
The first statement follows from the result in  \cite{ChapmanFlorenceMcKinnie:2023} mentioned just above the corollary. The second statement follows from the same description of the algebra but using  \cite[Theorem 1]{ReichsteinVistoli:2018} with the group $G=\Z/p^m\Z$ we see that after a prime-to-$p$ extension, the cyclic extension descends to a field of transcendence degree 1.
\end{proof}

For pairs of linked quaternion algebras $(C_1,C_2) \in \operatorname{LCA}_{2,2}$, the converse holds true as well, i.e., if $\operatorname{ed}(C_1,C_2)=2$, then $C_1$ and $C_2$ descend to a field of transcendence degree 2, which is of $u$-invariant 4, and thus $C_1$ and $C_2$ are inseparably linked.

In \cite{Chapman:2024} it was shown that $\operatorname{ed}_p(C_1,C_2)\leq 2$ for $(C_1,C_2) \in \operatorname{LCA}_{p,2}$ if and only if $C_1 \otimes L$ and $C_2 \otimes L$ are inseparably linked for some prime-to-$p$ extension $L/F$.

\begin{ques}
Does $\operatorname{ed}(C_1,C_2)=2$ for $(C_1,C_2) \in \operatorname{LCA}_{p,2}$ imply that $C_1$ and $C_2$ are inseparably linked?
\end{ques}


\begin{prop}\label{seqquat}
Both the essential dimension and the essential 2-dimension of $\operatorname{LCA}_{2,n}$ are equal to $n+1$.
\end{prop}

\begin{proof}
By \cite{ChapmanLevin:2025}, there is a well-defined invariant mapping any $n$-tuple of linked quaternion algebras $([\alpha,\beta_1)_{2,F},\dots,[\alpha,\beta_n)_{2,F}) \in \operatorname{LCA}_{2,n}(F)$ to the quadratic $(n+1)$-fold Pfister form $\langle \! \langle \beta_1,\dots,\beta_n,\alpha ] \!]$. In the generic case, the image is an anisotropic quadratic $(n+1)$-fold Pfister form, so it cannot descend to a field of transcendence degree smaller than $n+1$ over $k$. This is because $k$ is algebraically closed, and every field $E$ of transcendence degree $\leq n$ over $k$ has $u$-invariant at most $2^n$, and then all quadratic $(n+1)$-fold Pfister forms are isotropic. Therefore, the essential dimension is no less than $n+1$, and since it is clearly no greater than $n+1$, this is its exact value.

To complete the picture for the essential 2-dimension, it is enough to mention that anisotropic quadratic forms remain anisotropic under odd extensions.
\end{proof}
\begin{ques}
Is the essential ($p$-)dimension of $\operatorname{LCA}_{p,n}$ equal to $n+1$ for any prime $p$?
\end{ques}

The invariant $\langle \! \langle \beta_1,\dots,\beta_n,\alpha ] \!]$ in the proof of Proposition \ref{seqquat} of the $n$-tuple of linked quaternion algebras can also be taken to be the symbol $\alpha \frac{d \beta_1}{\beta_1} \wedge \dots \wedge \frac{d\beta_n}{\beta_n}$ in $H_2^{n+1}(F)$.
One could hope that a similar phenomenon holds for odd primes, but unfortunately, the symbol $\alpha \frac{d \beta_1}{\beta_1} \wedge \dots \wedge \frac{d\beta_n}{\beta_n}$ in $H_p^{n+1}(F)$ is not an invariant of the sequence $([\alpha,\beta_1)_{p,F},\dots,[\alpha,\beta_n)_{p,F}) \in \operatorname{LCA}_{p,n}(F)$ for $p>2$ as shown in \cite[Remark 3.6]{Chapman:2023}.

\section{The $\operatorname{Dec}_{p^m,n}$ functor}\label{dec}

In this section, we say a few words about the  functor $\operatorname{Dec}_{p^m,n}$ mapping each field $F$ containing $k$ to the subset of algebras in $\operatorname{Alg}_{p^{mn},p^m}(F)$ that are isomorphic to tensor products of $n$ cyclic algebras of degree $p^m$ over $F$.
The essential dimension and $p$-dimension of this functor were studied in \cite{McKinnie:2017} for $m=1$, and it was concluded that they are both equal to $n+1$.

\begin{thm}
The essential $p$-dimension of  $\operatorname{Dec}_{p^m,n}$ is $n+1$ for all positive integers $m,n$.
\end{thm}

\begin{proof}
There is a surjective map from $\operatorname{Dec}_{p^m,n}(F)$ to $\operatorname{Dec}_{p,n}(F)$ given by mapping each $[(\omega_1,\dots,\omega_m),\beta)_{p^m,F}$ to $[\omega_1,\beta)_{p,F}$. Hence $\ed_p(\operatorname{Dec}_{p^m,n}) \geq \ed_p(\operatorname{Dec}_{p,n}) = n+1$ by \cite{McKinnie:2017}.
On the other hand, by \cite[Theorem 1]{ReichsteinVistoli:2018}, the essential $p$-dimension of $(\mathbb{Z}/p^m \mathbb{Z})^{\times n}$ over fields of characteristic $p$ is 1, which means that every $n$-tuple $(\tau_1,\dots,\tau_n)$ of classes in $H_{p^m}^1(F)$ descends to a field $L$ of transcendence degree 1 over $k$. Hence a tensor product of $n$ cyclic algebras $[\tau_1,\beta_1)_{p^m,F} \otimes\cdots\otimes [\tau_n,\beta_n)_{p^m,F}$ of degree $p^m$ over a given $F$ descends to $L(\beta_1,\ldots,\beta_n)$, showing that the essential dimension is at most $n+1$.
\end{proof}

See Section \ref{PrimeExponent} for more information on $\operatorname{Dec}_{p,n}$.


\section{Cohomology Groups}\label{cohomology}

In the good characteristic case, Galois Cohomology gives a description of the $n$-torsion part of the Brauer group; ${_n\Br}(F) \cong H^2(F,\mu_n)$. In the bad characteristic case, the analogous cohomological construction is Kato-Milne cohomology for prime-torsion (see \cite{Kato:1982}) and Izhboldin groups in $p^m$-torsion (see \cite{Izhboldin:2000} and \cite{AravireJacobORyan:2018}). As pointed out to us by J.~P.~Tignol, though they are called ``Izhboldin groups", they were previously known to Kato and used by him (see \cite[Section 2, Corollary 4 to Proposition 2]{Kato:1980}). Both constructions utilize differential forms and are described here. 

The construction of Kato-Milne cohomology is as follows: we start with a field $F$ of $\operatorname{char}(F)=p$.
Take $\Omega_F^n$ the ring of $n$-fold differential forms (i.e., sums of terms of the form $a_1 db_1\wedge \dots \wedge db_n$ subject to the relations $d(bc)=bdc+cdb$ and $db\wedge db=0$), and consider the map $\wp : \Omega_F^n \rightarrow \Omega_F^n/d \Omega_F^{n-1}$ defined by $a \frac{db_1}{b_1} \wedge \dots \wedge\frac{db_n}{b_n} \mapsto (a^p-a) \frac{db_1}{b_1} \wedge \dots \wedge\frac{db_n}{b_n}$. The group $H_p^{n+1}(F)$ is the cokernel of this map, and it is the analogous group to $H^{n+1}(F,\mu_p^{\otimes n})$ in the more common case of fields $F$ of characteristic not $p$ containing primitive $p$th roots of unity (see \cite[Section 101]{EKM}).
In particular, for $p=2$, $H_2^{n+1}(F)$ is isomorphic to $I_q^{n+1}(F)/I_q^{n+2}(F)$, where $I_q^{n+1}(F)$ is the subgroup of the Witt group of nonsingular quadratic forms over $F$ generated by generalized $(n+1)$-fold Pfister forms. The isomorphism is given by $a \frac{db_1}{b_1}\wedge \dots \wedge\frac{db_n}{b_n} \mapsto \langle \! \langle b_1,\dots,b_n,a]\!]$ (see also \cite[Section 16]{EKM}). For $n=1$, $H_p^2(F)$ is isomorphic to ${_p\Br}(F)$, and the isomorphism is given by $a \frac{db}{b} \mapsto [a,b)_{p,F}$ (see Section \ref{primedegree}). The symbol length of a class in $H_p^{n+1}(F)$ is the number of terms of the form $a \frac{db_1}{b_1} \wedge \dots \wedge\frac{db_n}{b_n}$ needed to express it, and this number is always bounded from above by $\binom{r}{n}$ when $r$ is the $p$-rank of $F$ (see \cite[Corollary 3.4]{ChapmanMcKinnie:2020}).
Recall that when $F$ is a finitely generated field containing $k$ (and thus $\operatorname{char}(F)=p$), $F$ is a vector space over $F^p$ of dimension $p^r$ for some $r<\infty$, and this $r$ is called the ``$p$-rank" of $F$, and that when $F$ is of transcendence degree $n$ over $k$, its $p$-rank is $n$ (see \cite[Chapter V, Section 16.6, Corollary 3]{Bourbaki}).
Knowing the symbol length may say something about the decomposability of the algebra, which is the trick applied in \cite{ChapmanMcKinnie:2020} to provide lower bounds for the essential dimension of $\operatorname{Alg}_{p^n,p}$.

The group $\nu(n)_F$, defined to be the kernel of $\wp$, is the one corresponding to Milnor $K$-theory (i.e., the group $K_n(F)/pK_n(F)$), and it is the group $H^n(F,\mu_p^{\otimes n})$ in the language of \cite[Section 101]{EKM}. For $p=2$, this group is isomorphic to $I^n F/I^{n+1} F$ where $I^n F$ is the group generated by generalized bilinear $n$-fold Pfister forms inside the additive Witt group $W F$.
The Witt group is actually a ring, because the tensor product is defined, and more generally, there is a cup product that maps pairs $(A,B) \in \nu(\ell)_F \times \nu(m)_F$ to $A \wedge B \in \nu(\ell+m)_F$. Similarly, there is a cup product that maps pairs $(A,B) \in H_p^\ell(F) \times \nu(m)_F$ to $A \wedge B \in H_p^{\ell+m}(F)$. But note that since $H_p^m(F)$ and $\nu(m)_F$ need not be isomorphic, there is no cup product for pairs in $H_p^\ell(F) \times H_p^m(F)$.

In the good characteristic case, the cup product for pairs in $H^\ell(F,\mu_p) \times H^m(F,\mu_p)$ is defined, and it is applied (see \cite[Example 3.7]{Merkurjev:2013}) to produce the lower bound of $2n$ for the essential dimension of $\operatorname{Alg}_{p^n,p}$: there is a well-defined map from tensor products $(\alpha_1,\beta_1)_{p,F}\otimes \dots \otimes (\alpha_n,\beta_n)_{p,F}$ of $n$ cyclic algebras of degree $p$ over a field $F$ of characteristic 0 containing primitive $p$th roots of unity to the cup product of these algebras $(\alpha_1,\beta_1)\cup\dots \cup (\alpha_n,\beta_n)=(\alpha_1,\beta_1,\dots,\alpha_n,\beta_n)\in H^{2n}(F,\mu_p)$. Since the cohomological $p$-dimension of a field $E$ with $k \subseteq E \subseteq F$ is bounded from above by the transcendence degree of $E$ over $k$, the lower bound readily follows.

We cannot repeat the same argument in the bad characteristic case, because there is no cup product. However, the lower bound for the essential dimension of $\operatorname{Alg}_{p^n,p}$ in the good case mentioned above is $2n$ which corresponds to the essential dimension of $\operatorname{Dec}_{p,n}$ in this case. Similarly, in \cite{McKinnie:2017}, a lower bound of $n+1$ was produced for the essential dimension of $\operatorname{Alg}_{p^n,p}$ in the bad case, which again coincides with the essential dimension of $\operatorname{Dec}_{p,n}$ in this case. See Proposition \ref{Algpnp} for an explanation of this fact.

The Izhboldin groups $H_{p^m}^{n+1}(F)$ can be obtained in a similar way as a quotient of $W_m \Omega^n_F$ (\cite[Theorem 2.27]{AravireJacobORyan:2018}), which is the ring of truncated Witt vectors of length $m$ over $\Omega^n_F$ (see \cite[Chapter 8, Section 10]{Jacobson}), but a neater description in terms of generators and relations is the following (\cite{Izhboldin:2000}): take the additive group $W_m(F) \otimes \underbrace{F^\times \otimes \dots \otimes F^\times}_{n \ \text{times}}$ modulo the relations
\begin{itemize}
\item $(\omega^p-\omega) \otimes b_1 \otimes \dots \otimes b_n=0$, 
\item $(0\dots0,a,0,\ldots,0) \otimes a \otimes b_2 \otimes \dots \otimes b_n=0$, and 
\item $\omega \otimes b_1 \otimes \dots \otimes b_n=0$ where $b_i=b_j$ for some $i \neq j$,
\end{itemize}
where for each $\omega=(\omega_1,\dots,\omega_m)$, $\omega^p$ stands for $(\omega_1^p,\dots,\omega_m^p)$. Again, for $n=1$, these groups describe the $p^m$-torsion of the Brauer group, i.e., $\operatorname{H}_{p^m}^2(F) \cong \operatorname{Br}_{p^m}(F)$ with the isomorphism given by $\omega \otimes b \mapsto [\omega,b)_F$, where $[\omega,b)_F$ stands for the cyclic algebra described in Section \ref{sequence}. There is a natural embedding of $H_{p^m}^{n+1}(F)$ into $H_{p^{m+t}}^{n+1}(F)$ given by $(\omega_1,\dots,\omega_m) \otimes \beta_1 \otimes \dots \otimes \beta_n \mapsto (\underbrace{0,\dots,0}_{t \ \text{times}},\omega_1,\dots,\omega_m) \otimes \beta_1 \otimes \dots \otimes \beta_n$, and an epimorphism from $H_{p^{m+t}}^{n+1}(F)$ to $H_{p^t}^{n+1}(F)$ given by $(\omega_1,\dots,\omega_{m+t}) \otimes \beta_1 \otimes \dots \otimes \beta_n \mapsto (\omega_1,\dots,\omega_{t}) \otimes \beta_1 \otimes \dots \otimes \beta_n$, and together they form a short exact sequence
$$0 \rightarrow H_{p^m}^{n+1}(F) \rightarrow H_{p^{m+t}}^{n+1}(F) \rightarrow H_{p^t}^{n+1}(F) \rightarrow 0.$$

The group $H_{p^m}^1(F)$ is of particular interest because it describes the degree $p^m$ cyclic extensions of $F$, and is in fact $W_m F/\wp(W_m F)$. The essential dimension of the functor $H_{p^m}^1$ is the same as that of the finite group $\Z/p^m\Z$ and is unknown in general, except for the cases of $m=1$ and $2$, for which it is 1 and 2 respectively. See \cite{Ledet:2004} for further details, where it is also proven that $\operatorname{ed}(\Z/p^m\Z) \leq m$ in general, and an equality between the two is conjectured. 
However, the essential $p$-dimension of the functor $H_{p^m}^1$ is known to be 1 (see \cite{ReichsteinVistoli:2018}), which provides for an interesting case where the essential $p$-dimension is strictly smaller than the essential dimension.

Here is a small observation that gives a simple explanation for why the essential $p$-dimension of $H_{p^2}^1$ is 1:
\begin{prop}\label{p2}
Given $(a,b) \in H^1_{p^2}(F)$, there exists a prime-to-$p$ extension $L$ of $F$ of degree at most $p^2-p+1$ such that $(a,b)_L=(c,0)$ for a suitable $c\in L$.
\end{prop}

\begin{proof}
We are looking for an element $x$ in an extension of $F$ for which the second slot of $(a,b)+(x^p,0)-(x,0)$ is zero.
When $p$ is odd, the arithmetic of truncated Witt vectors of length 2 yields the following:
$$(a,b)+(x^p,0)-(x,0)=\left(a+x^p,b-\sum_{t=1}^{p-1} (-1)^{t-1} t^{-1} x^{tp}a^{p-t}\right)+(-x,0)$$
$$=\left(a+x^p-x,b-\sum_{t=1}^{p-1} (-1)^{t-1} t^{-1}(x^{tp}a^{p-t}+(a+x^p)^{t}(-x)^{p-t})\right).$$
The second slot is thus a monic polynomial of degree $p^2-p+1$, and therefore has a root in some prime-to-$p$ extension $L$ of $F$ of degree at most $p^2-p+1$.
When $p=2$,
$$(a,b)+(x^2,0)-(x,0)=(a+x^2,b+ax^2)+(x,x^2)$$
$$=(a+x^2+x,b+ax^2+(a+x^2)x).$$
The second slot is thus a monic cubic polynomial, and therefore has a root either in $F$ or in a cubic extension $L$ of $F$ (and here $p^2-p+1=3$).
\end{proof}

\begin{ques}
Can the argument in Proposition \ref{p2} be extended to show that $\ed_p(H^1_{p^m})=1$ for $m>2$?
\end{ques}

\section{Algebras of prime degree}\label{primedegree}

Fix a prime degree $p$ and a field $F$ of $\operatorname{char}(F)=p$.
Recall a cyclic algebra of degree $p$ over $F$ has a particularly nice presentation as it is generated by two elements $i$ and $j$ such that $i^p-i=\alpha$ for some $\alpha \in F$, $j^p=\beta$ for some $\beta \in F^\times$
and $jij^{-1}=i+1$. This algebra is denoted by the symbol $[\alpha,\beta)_{p,F}$.
The extension $F[i]/F$ is either a cyclic field extension of degree $p$ with Galois group $\langle \sigma \rangle$ where $\sigma(i)=i+1$ or $F \times \dots \times F$, an \'etale extension. In both cases, we can define the norm $N_\alpha : F[i] \rightarrow F$ by $x \mapsto x \sigma (x)\cdots \sigma^{p-1}(x)$ where $\sigma(i)=i+1$.
The algebra $[\alpha,\beta)_{p,F}$ is ``split", i.e., isomorphic to $M_p(F)$, when $\beta=N_\alpha(y)$ for some $y \in F[i]$. If no such $y$ exists, then $[\alpha,\beta)_{p,F}$ is a division algebra.
Different symbols can represent the same isomorphism class. In particular $[\alpha,\beta)_{p,F}=[\alpha+\lambda^p-\lambda+x^p\beta^n,\beta)_{p,F}$ for any choice of $\lambda,x \in F$ and $n\in \{1,\dots,p-1\}$, and $[\alpha,\beta)_{p,F}=[\alpha,\beta N_\alpha(y))_{p,F}$ for any $y \in F[i]$ (note that we use $=$ also for $\cong$ as is common in the literature).

It is still unknown whether  every algebra of degree $p$ over $F$ is cyclic. This is known when $p=2$ or $3$, but open for $p\geq 5$. In the case of $p$-special fields, however, this is known to be true. 

\begin{thm}[\cite{Albert:1968}]
If $F$ is a $p$-special field, then every central simple algebra of degree $p$ over $F$ is cyclic.
\end{thm}

\begin{proof}
This is an immediate result of Theorem \cite[Chapter IV, Theorem 32]{Albert:1968}. However, we include an alternative proof here using the Nullstellensatz for $p$-special fields.

Consider $A$ of degree $p$ over $F$.
A necessary and sufficient condition for $A$ to be cyclic is the inclusion of a non-central element $y$ satisfying $y^p \in F$ (\cite[Lemma 10]{Albert:1968}).
Recall that there are the characteristic coefficients $f_1,\dots,f_{p-1},f_p$ from $A$ to $F$, and that each $x\in A$ satisfies $x^p-f_1(x) x^{p-1}-\dots-f_{p-1}(x) x-f_p(x)=0$.
It is therefore enough to find a nontrivial solution to the system
$f_1(x)=f_2(x)=\dots=f_{p-1}(x)=0$ over a subspace of $A$ consisting of noncentral elements, and then the solution $y$ will be a noncentral element satisfying $y^p \in F$.
Since the center of $A$ is a subspace of dimension 1, there exists a subspace $V$ (satisfying $A=V\oplus F$) consisting of noncentral elements (apart from zero) of dimension $p^2-1$. By the Nullstellensatz for $p$-special fields (\cite[Theorem 1.13]{Pfister:1995}), there exists a nontrivial solution to the system $f_1(x)=f_2(x)=\dots=f_{p-1}(x)=0$ over $V$, and that settles the issue.
\end{proof}

The following result is analogous to {\cite[Lemma 8.5 (7)]{ReichsteinYoussin:2000}} which was stated in the good characteristic case:
\begin{cor}\label{Algpp}
The essential $p$-dimension of $\operatorname{Alg}_{p,p}$ is $2$ when the base field $k$ is algebraically closed and of characteristic $p$.
\end{cor}

\begin{proof}
Let $F$ be a field containing $k$. Every cyclic $F$-algebra, $[\alpha,\beta)$ is defined over $k(\alpha,\beta)$, a field of transcendence degree at most 2 over the base field, $k$.  By the previous theorem, every algebra of degree $p$ over $F$ becomes cyclic under some prime-to-$p$ extension $L/F$, and therefore, the essential $p$-dimension of $\operatorname{Alg}_{p,p}(F)$ is at most 2.
In order to show that it is exactly 2, it is enough to point out one cyclic division algebra of degree $p$, because the transcendence degree bounds the cohomological $p$-dimension, and for the existence of division cyclic algebras, it must be at least 2.
By \cite[Remark]{ChapmanChapman:2017} (using valuation theory, see \cite{TignolWadsworth:2015} for background), the generic cyclic algebra $[\alpha,\beta)_{p,F}$ is division for $F=k(\alpha,\beta)$. Since it remains division under prime-to-$p$ extension, the statement readily follows.


\end{proof}


\section{Algebras of Prime Exponent}\label{PrimeExponent}

By \cite[Theorem 9.1.4]{GilleSzamuely:2006}, which dates back to Teichm\"uller, ${_p\Br}(F)$ is generated by cyclic algebras of degree $p$ for fields $F$ of $\operatorname{char}(F)=p>0$.
Tensor products of cyclic algebras of degree $p$ therefore have a special place in the theory of central simple algebras. 

As mentioned in Section \ref{cohomology}, there is a description of ${_p\Br}(F)$ in terms of differential forms. Recall that $\Omega_F$ is the ring of differential forms over $F$, i.e., sums of terms of the form $a db$, where $a,b \in F$, and $d(bc)=cdb+bdc$. The group ${_p\Br}(F)$ is isomorphic to the cokernel of $\wp$, where $\wp : \Omega_F \rightarrow \Omega_F/dF$ is given by $a\frac{db}{b} \mapsto (a^p-a)\frac{db}{b}$. The isomorphism takes $a\frac{db}{b}$ to $[a,b)_{p,F}$ and, as the argument in Section \ref{cohomology} with $n=1$ shows, the $p$-rank of $F$ bounds the symbol length of ${_p\Br}(F)$, i.e., the maximal number of cyclic algebras of degree $p$ required in order to express any given class in ${_p\Br}(F)$ is bounded from above by the $p$-rank of $F$.

Consider the functor $\operatorname{Dec}_{p,n}$ mapping each field $F$ containing $k$ to the set of algebras of degree $p^n$ that decompose as tensor products of $n$ cyclic algebras of degree $p$ over $F$. The fact that $\operatorname{ed}((\mathbb{Z}/p\mathbb{Z})^{\times n})$ is 1 over fields of characteristic $p$ (\cite{Ledet:2004}) implies that $\operatorname{ed}(\operatorname{Dec}_{p,n})\leq n+1$. 
Here is a direct proof of this fact without using the result on $\operatorname{ed}((\mathbb{Z}/p\mathbb{Z})^{\times n})$:
\begin{prop}
The essential dimension of $\operatorname{Dec}_{p,n}$ is bounded from above by $n+1$.
\end{prop}
\begin{proof}
In \cite{ChapmanQueguiner:2024} it was shown that every tensor product of $n$ cyclic algebras of degree $p$ can be written as
\begin{eqnarray*}\label{neat}
[\alpha_1,\alpha_2)_{p,F} \otimes [\alpha_2,\alpha_3)_{p,F} \otimes \dots \otimes [\alpha_n,\alpha_{n+1})_{p,F}.
\end{eqnarray*}
This observation is based on the rule $[\alpha,\beta)_{p,F}=[\alpha(\beta+x^p)\beta^{-1},\beta+x^p)_{p,F}$, which means that each $[\alpha,\beta)_{p,F} \otimes [\gamma,\delta)_{p,F}$ can be modified into $[\alpha(\beta+x^p)\beta^{-1},\beta+x^p)_{p,F} \otimes [\gamma+x^p-x,\delta)_{p,F}$ where $x=\gamma-\beta$ (i.e., $\gamma+x^p-x=\beta+x^p$), and the statement follows by induction.
\end{proof}

In \cite{McKinnie:2017}, the generic tensor product of $n$ cyclic algebras of degree $p$ was used to bound the essential dimensions of both $\operatorname{Dec}_{p,n}$ and $\operatorname{Alg}_{p^n,p}$ from below by $n+1$. 
For $p=2$ and $n=3$, this improved the lower bound produced in \cite{Baek:2011} for $\operatorname{Alg}_{8,2}$ from 3 to 4.
We recall here briefly the argument made in \cite{ChapmanMcKinnie:2020} for $n+1$ being a lower bound for $\operatorname{Alg}_{p^n,p}$:

\begin{prop}\label{Algpnp}
The essential $p$-dimension of $\operatorname{Alg}_{p^n,p}$ is at least $n+1$.
\end{prop}
\begin{proof}
When $n=1$ the essential $p$-dimension is 2 by Corollary \ref{Algpp}. Suppose $n\geq 2$.
The case of $p=n=2$ is also easy, see Proposition \ref{Alg42}.
When either $n\geq 3$ or $p \geq 3$, indecomposable algebras of degree $p^n$ and exponent $p$ exist and remain indecomposable after any prime to $p$ extension (\cite{Karpenko:1995}). If the $p$-essential dimension were $\leq n$, then a prime to $p$ extension of the algebra would descend to a field of $p$-rank $\leq n$, and therefore have symbol length at most $n$. That would mean the algebra decomposes as a tensor product of $\leq n$ cyclic algebras of degree $p$. Hence, indecomposable algebras of degree $p^n$ and exponent $p$ are of $p$-essential dimension at least $n+1$.
\end{proof}

Since the essential dimension of the generic totally decomposable algebras is $n+1$, one might expect that indecomposable algebras to be of greater essential dimension.

\begin{ques}
What is the essential dimension of the indecomposable algebras of degree $p^n$ and exponent $p$ constructed in \cite{Karpenko:1995}? 
\end{ques}

This lower bound of $n+1$ for $\operatorname{ed}(\operatorname{Alg}_{p^n,p})$ is the current state of the art. For upper bounds, in \cite[Theorem 1.3]{GaribaldiGuralnick:2016}, an upper bound of $2n^2-3n-6$ was provided for $\operatorname{ed}(\operatorname{Alg}_{2n,2})$ when $n$ is even, $n>4$ and $\operatorname{char}(k)=2$. Replacing $n$ with $2^{n-1}$ yields $\operatorname{ed}(\operatorname{Alg}_{2^n,2})\leq 2^{2n-1}-3\cdot 2^{n-1}-6$ for $n\geq 4$. 
For results on $\operatorname{Alg}_{4,2}$ and $\operatorname{Alg}_{8,2}$, see Sections \ref{d4} and \ref{d8}. In \cite{Baek:2012}, it was proven that $\operatorname{ed}_2(\operatorname{Alg}_{2^n,2})\leq 2^{2n-2}$.

\section*{Acknowledgments}

The authors are greatly indebted to Jean-Pierre Tignol, Uzi Vishne, and the anonymous referees for their helpful remarks on earlier versions of the manuscript. 

\bibliographystyle{abbrv}
\def\cprime{$'$}

\end{document}